\documentclass[11pt]{article}
\usepackage[T1]{fontenc}
\usepackage{mathrsfs}
\usepackage{amsmath}
\usepackage{bbm,color}
\usepackage{amsfonts}

\makeatletter

\@addtoreset{equation}{section}
\makeatother

\usepackage{amssymb}
\usepackage{amsthm}
\usepackage[utf8]{inputenc}
\usepackage[hidelinks,bookmarksnumbered=true]{hyperref}

\newtheorem{theorem}{Theorem}[section]
\newtheorem{proposition}[theorem]{Proposition}
\newtheorem{lemma}[theorem]{Lemma}
\newtheorem{corollary}[theorem]{Corollary}
\theoremstyle{definition}

\newtheorem{remark}[theorem]{Remark}
\newtheorem{example}[theorem]{Example}

\newcommand{\C}{\mathbb C}
\newcommand{\R}{\mathbb R}
\newcommand{\Q}{\mathbb Q}
\newcommand{\Z}{\mathbb Z}

\newcommand{\D}{\mathbb D}
\newcommand{\T}{\mathbb T}
\newcommand{\Chat}{\widehat{\C}}
\newcommand{\id}{\operatorname{id}}
\newcommand{\Cent}{\operatorname{Cent}_{\mathrm{an}}}
\newcommand{\Centfor}{\operatorname{Cent}_{\mathrm{for}}}
\newcommand{\Diff}{\operatorname{Diff}(\C,0)}
\newcommand{\Diffhat}{\widehat{\operatorname{Diff}}(\C,0)}
\newcommand{\Bcal}{\mathcal B}
\newcommand{\ord}{\operatorname{ord}}
\newcommand{\im}{\operatorname{im}}

\newcommand{\iter}[2]{#1^{\circ #2}}
\newcommand{\cyc}[1]{\langle #1\rangle}
\newcommand{\disc}[1]{\D_{#1}}

\allowdisplaybreaks[2]
\title{Analytic centralisers of quadratic germs\\and the Brjuno dichotomy}

\author{{\sc Wen Si}$$\thanks{E-mail: {\tt siwenmath@sdu.edu.cn}}
\\
	 {\small School of Mathematics, Shandong University,}\\
	{\small Jinan, Shandong 250100, People's Republic of China}
}

\date{}

\begin{document}
\maketitle

\begin{abstract}
For $Q_\alpha(z)=e^{2\pi i\alpha}z+z^2$ with $\alpha$ irrational, we prove
that the analytic centraliser of $Q_\alpha$ is trivial whenever $Q_\alpha$
is not analytically linearizable. Together with the quadratic
Brjuno--Yoccoz theorem this gives a complete dichotomy, for the centraliser
is infinite cyclic at the non-Brjuno parameters and conjugate to the group
of nonzero complex dilations at the Brjuno parameters. The result extends
the dense-parameter conclusion of Avila--Cheraghi--Eliad to every
nonlinearizable member of the quadratic family. The proof starts from
Inou's graph-globalization construction, where we retain the original
neutral fixed point on a Riemann surface and verify the branch topology and
finite-valence estimate, including multiplicities. Nonlinearizability,
through the equality case of Schwarz's lemma, forces the surface to be the
plane, and a direct quadratic growth estimate then makes the two entire
projections polynomial, so that the established polynomial commutant
classification yields the integral-iterate conclusion. We also determine
the convergence set of the formal commutants and classify the analytic
iterative roots of every nonzero iterate.
\end{abstract}

\begin{quote}

\end{quote}

\vskip 2mm
\tableofcontents

\newpage
\section{Introduction and main results}\label{sec:introduction}

\subsection{Local analytic symmetries}
For an invertible holomorphic germ
\begin{equation*}
 f(z)=\lambda z+O(z^2),\qquad \lambda\ne0,
\end{equation*}
its local analytic centraliser is
\begin{equation}\label{eq:centraliser}
 \Cent(f)=\{h\in\Diff:h\circ f=f\circ h\}.
\end{equation}
All identities in \eqref{eq:centraliser} are identities of germs at $0$.
Here $\Diff$ denotes the group of invertible holomorphic germs at $0$,
namely those of the form $h(z)=a_1z+\sum_{n\ge2}a_nz^n$ with
$a_1=h'(0)\ne0$.
The inverse function theorem determines a unique inverse germ fixing $0$.
Write $\iter f k$ for the $k$-th iterate of $f$, with $\iter f 0=\id$,
$\iter f{k+1}=f\circ\iter f k$ for $k\ge0$, and
$\iter f k=(\iter f{-k})^{-1}$ for $k<0$. Consequently
\begin{equation}\label{eq:integer-group}
 \cyc f=\{\iter f k:k\in\Z\}\subset\Cent(f),
 \qquad \iter f k\circ\iter f\ell=\iter f{k+\ell}.
\end{equation}
Following Avila--Cheraghi--Eliad~\cite{ACE}, the centraliser is called
\emph{trivial} when equality holds in \eqref{eq:integer-group}. Thus
``trivial'' means that no symmetries occur beyond the integral iterates,
not that the identity is the only symmetry.

The global problem of deciding which rational maps commute is classical and
goes back to Ritt~\cite{Ritt}, while the polynomial factorization used
below is the precise form proved by Pakovich~\cite{Pakovich}.
Those results describe maps on the whole sphere, whereas the objects in
\eqref{eq:centraliser} are germs, and the difference is essential.
Such a germ need not extend to a single-valued function on $\C$, and a
negative iterate generally has no such extension, so the global theory
cannot be transferred literally. Even so, once $f$ is analytically
linearizable, every nonzero dilation becomes a local symmetry and the
centraliser is large. Write $\C^*=\C\setminus\{0\}$ for the nonzero
complex numbers. More precisely, if
\begin{equation*}
 \psi\circ f\circ\psi^{-1}=R_\lambda,
 \qquad R_\lambda(w)=\lambda w,
\end{equation*}
then
\begin{equation*}
 \psi^{-1}\circ R_\mu\circ\psi\in\Cent(f)
 \qquad(\mu\in\C^*).
\end{equation*}
One might hope that nonlinearizability forces the centraliser to shrink,
but this fails in general, for P\'erez Marco constructed nonlinearizable
irrationally indifferent germs with uncountably many analytic
symmetries~\cite{PM}. Thus neither the modulus of the multiplier nor
nonlinearizability by itself determines the centraliser in the unrestricted
class of holomorphic germs.

We consider the quadratic family
\begin{equation}\label{eq:quadratic-family}
 f_\lambda(z)=\lambda z+z^2=Q_\alpha(z),
 \qquad \lambda=e^{2\pi i\alpha},\quad \alpha\in\R\setminus\Q.
\end{equation}
Avila--Cheraghi--Eliad proved triviality for every rational parameter and
for a dense set of irrational parameters~\cite[Theorems~1.1--1.2]{ACE}.
Their irrational examples are obtained by successive perturbations of the
parabolic parameters. The question addressed here is the all-parameter
strengthening for the nonlinearizable members of \eqref{eq:quadratic-family}. We obtain the following theorem.
\begin{theorem}[Quadratic centraliser rigidity]\label{thm:main}
Let $\alpha\in\R\setminus\Q$ with $\lambda=e^{2\pi i\alpha}$  and
$f(z)=\lambda z+z^2$. If $f$ is not analytically linearizable at $0$, then
\begin{equation*}
 \Cent(f)=\cyc f=\{\iter f k:k\in\Z\}.
\end{equation*}
Moreover, for each $h\in\Cent(f)$, the integer $k$ in $h=\iter f k$ is unique.
\end{theorem}

Let $p_n/q_n$ be the ordinary continued-fraction convergents of $\alpha$.
We use the convention $q_{-1}=0$, $q_0=1$ and define the Brjuno set as follows
\begin{equation*}
 \Bcal=\left\{\alpha\in\R\setminus\Q:
   \sum_{n\ge0}\frac{\log q_{n+1}}{q_n}<\infty\right\}.
\end{equation*}
The Brjuno--Yoccoz theorem for the quadratic family states
\begin{equation}\label{eq:BY}
 \alpha\in\Bcal
 \quad\Longleftrightarrow\quad
 Q_\alpha\text{ is analytically linearizable at }0.
\end{equation}
We use \eqref{eq:BY} as an established arithmetic theorem~\cite{Yoccoz}.
See also the quantitative treatment of quadratic linearization in
Buff--Ch\'eritat~\cite{BC}. The equivalence
\eqref{eq:BY} is not reproved here. Its necessity is specific to
this quadratic family and must not be assumed for an arbitrary germ
having the same multiplier.

\begin{corollary}[The Brjuno dichotomy]\label{cor:dichotomy}
For every irrational $\alpha$, put $\lambda=e^{2\pi i\alpha}$. Then
\begin{equation}\label{eq:dichotomy}
 \Cent(Q_\alpha)=
 \begin{cases}
 \{\psi^{-1}\circ R_\mu\circ\psi:\mu\in\C^*\},
   &\alpha\in\Bcal,\\[1mm]
 \{\iter {Q_\alpha} k:k\in\Z\},&\alpha\notin\Bcal,
 \end{cases}
\end{equation}
where  $\psi\circ Q_\alpha\circ\psi^{-1}=R_\lambda$ when $\alpha\in\Bcal$.
In particular,
\begin{equation}\label{eq:iff-trivial}
 \Cent(Q_\alpha)\text{ is trivial}
 \quad\Longleftrightarrow\quad \alpha\notin\Bcal.
\end{equation}
\end{corollary}

Our definition allows every nonzero multiplier for $h$. If one restricts
the centraliser to germs with multiplier of modulus one, the group
$\C^*$ in the first line of \eqref{eq:dichotomy} is replaced by the unit circle $\T=\R/\Z$.
For a nonlinearizable irrationally indifferent germ, a commuting
invertible germ necessarily has multiplier of modulus one. An
independent proof of this fact is given in
Lemma~\ref{lem:unit-multiplier}. Thus the non-Brjuno
conclusion agrees with the unit-multiplier convention for local germs
used in \cite{ACE}.

\subsection{Relation to previous work}
The dense-parameter theorem of Avila--Cheraghi--Eliad
\cite[Theorem~1.1]{ACE} is an existence and density result. Their
Theorem~1.2 supplies the parabolic parameters, while their Section~3
builds the irrational parameters by successive perturbation.
Theorem~\ref{thm:main} differs in kind, because it starts from an
arbitrary nonlinearizable member of the fixed quadratic family, so it
neither selects the parameter nor requires a subsequence of exceptionally
strong rational approximations. The all-parameter statement is therefore
a strengthening of their theorem rather than a restatement of a
conjecture explicitly formulated in \cite{ACE}. The open problem on
p.~887 of that paper asks for the optimal joint arithmetic condition for
simultaneous linearization of commuting maps, and since that problem is
related to Moser's simultaneous approximation approach~\cite{Moser}, it
is distinct from the classification of the centraliser of the particular
polynomial $Q_\alpha$.

A second distinction separates the classification of all germs with a
prescribed multiplier from the classification of the centraliser of a
specified polynomial. The large-centraliser examples of P\'erez
Marco~\cite{PM} show that nonlinearizability does not force cyclicity for
general germs, and Biswas describes nonlinearizable maximal abelian
groups by germs of hedgehogs~\cite[Theorems~3.9--3.10]{Biswas}. Such
structural results do not identify the symmetry group of each quadratic
member, which is why a geometric description of local invariant sets and
an exact calculation of the centraliser are different assertions.

Renormalization supplies complementary information about irrationally
indifferent dynamics. Cheraghi's topological trichotomy
\cite[Theorem~A]{Cheraghi}, for instance, describes postcritical
attractors for sufficiently high-type rotation numbers in the
Inou--Shishikura class, which includes the corresponding quadratic maps.
We neither remove the high-type hypothesis from that topological theorem
nor use its conclusion as an input, since the present argument classifies
analytic symmetries without assuming a description of a hedgehog, a
marked critical orbit, or the postcritical set.

The proof combines existing local-to-global and polynomial rigidity
methods. Its three published ingredients are the graph construction of
Inou~\cite[Theorem~2]{Inou}, the transcendental semiconjugacy obstructions
that trace back to Eremenko~\cite[Theorem~2]{Eremenko}, and the polynomial
factorization of Pakovich~\cite[Theorem~1.3]{Pakovich} that is used at the
last step. The contribution is the all-parameter quadratic conclusion,
obtained by retaining a nonlinearizable neutral fixed point through
globalization. The construction, the uniformization, and the polynomial
classification are not claimed as new theories, so the proofs below make
explicit how their hypotheses are met in the present neutral setting.

\subsection{Outline of the proofs}

The proof proceeds in four steps.

\textbf{Step 1 (globalization).}
The starting point is the local-to-global mechanism, which separates
the analytic continuation problem from the global polynomial commutant
problem. Given $h\in\Cent(f)$, choose a representative on a disc $U$ and
set
\begin{equation*}
 F=(f,f),\qquad \Psi_0(z)=(z,h(z)),\qquad
 \Gamma=\bigcup_{n\ge0}F^{\circ n}(\Psi_0(U)).
\end{equation*}
The set $\Gamma$ need not initially be a closed analytic subset of $\C^2$.
Different local branches above the same point are therefore retained
separately. This produces a Riemann surface $X$, a self-map $G$, and
holomorphic projections $\phi_1,\phi_2$ satisfying
\begin{equation*}
 \phi_i\circ G=f\circ\phi_i\quad(i=1,2),
 \qquad h=\phi_2\circ\phi_1^{-1}\text{ near a marked point }x_*.
\end{equation*}
This is the polynomial specialization of the desingularization
construction in Inou~\cite[Theorem~2]{Inou}. We give its branch-coordinate
implementation and prove the topological and local tracking properties
needed for the neutral fixed-point argument. In particular, with
$\deg_xG$ denoting local degree (fixed in Section~\ref{sec:surface}),
\begin{equation}\label{eq:overview-lift}
 G(x_*)=x_*,\qquad G'(x_*)=\lambda,
 \qquad \sum_{G(x)=y}\deg_x G\le4.
\end{equation}
The fixed point is an interior point of the surface and its local germ
is conjugate to $f$.

\textbf{Step 2 (uniformization).}
The first rigidity step identifies the surface. Nonlinearizability rules
out a hyperbolic surface (one whose universal cover is the unit disc),
and the uniformization argument then excludes
the remaining possibilities. Here $\simeq$ denotes biholomorphic
equivalence of Riemann surfaces, so that
\begin{equation}\label{eq:roadmap-one}
 f\text{ nonlinearizable}
 \Longrightarrow X\text{ nonhyperbolic}
 \Longrightarrow X\simeq\C.
\end{equation}
The first implication uses Schwarz's equality case.

\textbf{Step 3 (polynomiality).}
Once the surface is the plane, the projections and the self-map are
forced to be polynomial, and the self-map is forced to have degree two,
where $\deg$ denotes polynomial degree, $\#S$ denotes the cardinality
of a set $S$, and $\C[z]$ denotes the ring of polynomials in $z$.
\begin{align}
 \#G^{-1}(y)\le4
 &\Longrightarrow G\in\C[z],\quad 2\le\deg G\le4,\label{eq:roadmap-two}\\
 \phi_i\circ G=\phi_i^2+\lambda\phi_i
 &\Longrightarrow \phi_i\in\C[z],\quad\deg G=2.\label{eq:roadmap-three}
\end{align}
For the second implication, Eremenko's semiconjugacy theorem already
provides a general rigidity framework~\cite{Eremenko}, as used in
\cite[Section~3.2]{Inou} and \cite[Theorem~2.3]{JX}. Here the quadratic
degree permits the self-contained estimate
\begin{equation*}
 H(cr^2)\le H(r)^2
 \quad\Longrightarrow\quad H(r)\le Cr^K
\end{equation*}
for a modified maximum modulus of $\phi_i$.

\textbf{Step 4 (polynomial commutant classification).}
After affine normalization of $G$, the two projections become
polynomials commuting with $f$. The minimal-degree commutant theorem
of Pakovich~\cite[Theorem~1.3]{Pakovich}, together with a coefficient
calculation for affine maps, gives
\begin{equation*}
 \phi_i\circ A^{-1}=\iter f{n_i},\qquad
 h=\iter f{n_2-n_1}.
\end{equation*}
The global polynomial theorem is applied only after the projections have
been proved polynomial. It is not applied to the original local germ.

Local-to-global rigidity at repelling fixed points is developed further
by Ji--Xie~\cite[Section~2]{JX}. In the cited version, their
Theorem~2.2 records Inou's construction and their Theorem~2.3 records
the Eremenko obstruction. Their Theorem~2.1 additionally assumes a
repelling fixed point and a non-injective restriction. We do not apply
that theorem to a neutral point: instead,
Lemma~\ref{lem:schwarz} and Proposition~\ref{prop:plane} supply the
required exclusion of a hyperbolic surface and of a degree-one plane
lift. This is the point at which nonlinearizability replaces expansion
in the present argument.

\subsection{Organization}
The argument is organized in four sections.
Chapter~\ref{sec:globalize} builds the globalization of the local germ:
Section~\ref{sec:branches} supplies the analytic branch calculations and
Section~\ref{sec:surface} constructs the surface and proves
\eqref{eq:overview-lift}. Chapter~\ref{sec:unif-plane} identifies that
surface, giving \eqref{eq:roadmap-one}--\eqref{eq:roadmap-two}.
Chapter~\ref{sec:rigidity} then forces the projections to be iterates
of $f$: the growth estimate in Section~\ref{sec:growth} yields
\eqref{eq:roadmap-three}, and Section~\ref{sec:polynomial} records the
exact external polynomial factorization theorem used to classify the
polynomial commutants. Chapter~\ref{sec:conclude} completes the main
proof, identifies the globalized curve, and gives the formal convergence
and iterative-root consequences. The role of each hypothesis is explained by a remark placed where it is used: the quadratic degree in Remark~\ref{rem:quadratic-degree}, and the arithmetic and nonlinearizability inputs in Remarks~\ref{rem:arithmetic-qualitative} and~\ref{rem:nonlinearizability}.

\section{Analytic branches and the globalization surface}\label{sec:globalize}

\subsection{Analytic branches and primitive coordinates}\label{sec:branches}

Write $\disc r=\{z\in\C:|z|<r\}$. A branch at $a\in\C^2$ means a
reduced irreducible one-dimensional analytic germ based at $a$.
A primitive parametrization of such a branch is a holomorphic
parametrization that is one-to-one on a sufficiently small disc.
Its derivative may vanish at the centre, so that injectivity, rather
than an immersion at the centre, is the property needed in the sequel.
The normalization and separation facts are standard in local
analytic-curve theory, as in Chirka~\cite{Chirka} and in the discussion
preceding Theorem~2 in Inou~\cite{Inou}. We nevertheless supply direct
proofs for the parametrized branches that actually occur, so that no
global analytic-set structure on the forward-image union is presupposed.

\begin{lemma}[An explicit primitive parametrization]\label{lem:primitive}
Let
\begin{equation*}
 \Psi(t)=(A(t),B(t)),\qquad \Psi(0)=(a,b),
\end{equation*}
be a holomorphic germ with $A$ nonconstant. There exist a local
coordinate $s=s(t)$, an integer $\ell\ge1$, and a holomorphic map
$\nu$ such that
\begin{equation}\label{eq:primitive-factor}
 \Psi(t)=\nu(s(t)^\ell),\qquad s(0)=0,\quad s'(0)\ne0.
\end{equation}
The map $\nu$ parametrizes the reduced irreducible image branch, is
injective on a sufficiently small disc, and is an immersion off its
centre.
\end{lemma}
\begin{proof}
Write $\ord_a$ for the order of vanishing of a germ at a point $a$,
and let $\ord_0$ denote the case $a=0$. With this convention,
set $m=\ord_0(A-a)$. Then
\begin{equation*}
 A(t)-a=t^m u(t),\qquad u(0)\ne0.
\end{equation*}
After shrinking the disc, $u$ has a holomorphic $m$th root. The change
of variable
\begin{equation*}
 s(t)=t\,u(t)^{1/m}
\end{equation*}
satisfies $s'(0)=u(0)^{1/m}\ne0$ and yields
\begin{equation}\label{eq:s-series}
 A=a+s^m,\qquad B=b+\sum_{k\ge1}b_k s^k.
\end{equation}
Define
\begin{equation*}
 S=\{k\ge1:b_k\ne0\},\qquad
 \ell=\gcd(\{m\}\cup S),\qquad M=m/\ell,
\end{equation*}
where $\gcd$ of a set of integers means their greatest common
divisor, and $\ell=m$ when $S=\varnothing$. In the variable
$\tau=s^\ell$,
put
\begin{equation}\label{eq:nu}
 \nu(\tau)=\left(a+\tau^M,
    b+\sum_{k\in S}b_k\tau^{k/\ell}\right)
             =\bigl(a+\tau^M,B_*(\tau)\bigr).
\end{equation}
All exponents in \eqref{eq:nu} are integers. If the series in
\eqref{eq:s-series} converges for $|s|<r$, the second component of
\eqref{eq:nu} converges for $|\tau|<r^\ell$. This proves
\eqref{eq:primitive-factor}.

Suppose $\nu(\tau)=\nu(\sigma)$. Equality of first components gives
\begin{equation}\label{eq:root-comparison}
 \sigma^M=\tau^M.
\end{equation}
If $\tau=0$, then $\sigma=0$. If $\tau\ne0$, then
$\sigma=\zeta\tau$ for some $\zeta^M=1$. For each $\zeta\ne1$
with $\zeta^M=1$, set
\begin{equation*}
 D_\zeta(\tau)=B_*(\tau)-B_*(\zeta\tau)
 =\sum_{k\in S}b_k(1-\zeta^{k/\ell})\tau^{k/\ell}.
\end{equation*}
The equality
\begin{equation}\label{eq:primitive-gcd}
 \gcd\bigl(\{M\}\cup\{k/\ell:k\in S\}\bigr)=1
\end{equation}
implies $D_\zeta\not\equiv0$. Indeed, if $D_\zeta$ vanished identically
and $r$ were the order of $\zeta$, then $r$ would divide every integer
in \eqref{eq:primitive-gcd}, so that $r=1$, which contradicts
$\zeta\ne1$. Therefore
\begin{equation}\label{eq:isolated-difference}
 D_\zeta(\tau)=\tau^{m_\zeta}v_\zeta(\tau),
 \qquad v_\zeta(0)\ne0.
\end{equation}
Since there are finitely many such $\zeta$, one can choose
$\varepsilon>0$ so that all $v_\zeta$ are nonzero on $\disc\varepsilon$.
Equations \eqref{eq:root-comparison}--\eqref{eq:isolated-difference}
then give
\begin{equation*}
 \nu(\tau)=\nu(\sigma),\quad |\tau|,|\sigma|<\varepsilon
 \quad\Longrightarrow\quad \tau=\sigma.
\end{equation*}
Moreover,
\begin{equation*}
 \frac{\mathrm d}{\mathrm d\tau}(a+\tau^M)=M\tau^{M-1}\ne0
 \qquad(0<|\tau|<\varepsilon),
\end{equation*}
so $\nu$ is an immersion away from its centre.

We verify that the image is an analytic branch. Consider
\begin{equation}\label{eq:weierstrass-product}
 W(a+\tau^M,y)=\prod_{\zeta^M=1}\bigl(y-B_*(\zeta\tau)\bigr).
\end{equation}
Each coefficient on the right is invariant under $\tau\mapsto\xi\tau$
for $\xi^M=1$. Its convergent Taylor expansion therefore contains only
powers $\tau^{jM}$, so it is a convergent power series in
$x-a=\tau^M$. Thus $W(x,y)$ is holomorphic near $(a,b)$ and monic in
$y$. For $x$ near $a$, choose $\tau^M=x-a$.
Then
\begin{align}\label{eq:zero-set-param}
 W(x,y)=0
 &\Longleftrightarrow y=B_*(\zeta\tau)\text{ for some }\zeta^M=1\\
 &\Longleftrightarrow (x,y)=\nu(\zeta\tau)\text{ for some }\zeta^M=1.
 \nonumber
\end{align}
This identifies the germ of the zero set with the parametrized image.
The factors in \eqref{eq:weierstrass-product} are distinct over the
punctured parameter disc by \eqref{eq:isolated-difference}.

For irreducibility, suppose the image were a union of two proper
analytic curve germs $C_1\cup C_2$. Choose local defining functions
$q_i$ that vanish on $C_i$ but do not vanish identically on the
other component. Pullback gives
\begin{equation*}
 (q_1\circ\nu)(q_2\circ\nu)\equiv0.
\end{equation*}
The ring $\C\{\tau\}$ of convergent power series in $\tau$ is an integral domain, so one factor is
identically zero. Equation \eqref{eq:zero-set-param} then puts the
whole image in one proper component, contradicting the assumed
decomposition. The reduced image is consequently irreducible.
\end{proof}

\begin{lemma}[Uniqueness of the primitive coordinate]\label{lem:transition}
If $\nu_1$ and $\nu_2$ are primitive parametrizations of the same
branch at $a$, with $\nu_i(0)=a$, then, after shrinking their domains,
\begin{equation*}
 \nu_1=\nu_2\circ\theta,
 \qquad \theta(0)=0,\quad\theta'(0)\ne0,
\end{equation*}
for a biholomorphic germ $\theta$.
\end{lemma}
\begin{proof}
By equality of the image germs, the expression
\begin{equation*}
 \theta=\nu_2^{-1}\circ\nu_1
\end{equation*}
is defined on a sufficiently small punctured disc. Both parametrizations
are immersions there, so $\theta$ is holomorphic there. Choose
representatives $\nu_i$ injective on slightly larger closed discs.
A continuous injection from a compact space into the Hausdorff space
$\C^2$ is a homeomorphism onto its image, so
\begin{equation*}
 \nu_1(t)\longrightarrow a
 \quad\Longrightarrow\quad
 \nu_2^{-1}(\nu_1(t))\longrightarrow0.
\end{equation*}
Hence $\theta$ extends continuously to $0$ with value $0$, and its
singularity there is removable. Reversing the roles
of $\nu_1$ and $\nu_2$ gives a holomorphic extension of
$\eta=\nu_1^{-1}\circ\nu_2$. On punctured discs,
\begin{equation*}
 \eta\circ\theta=\id,
 \qquad\theta\circ\eta=\id.
\end{equation*}
The identity theorem extends these equations across $0$. Differentiating
at $0$ gives $\eta'(0)\theta'(0)=1$, proving the assertion.
\end{proof}

\begin{lemma}[Separation of distinct branches]\label{lem:separation}
Distinct irreducible analytic branches $C_1,C_2$ at $a\in\C^2$
have sufficiently small representatives whose intersection is
contained in $\{a\}$.
\end{lemma}
\begin{proof}
Choose a defining holomorphic germ $q$ for $C_2$ which does not vanish
identically on $C_1$, and let $\nu_1$ be a primitive parametrization
of $C_1$. Then
\begin{equation*}
 q\circ\nu_1(t)=t^s u(t),\qquad u(0)\ne0.
\end{equation*}
On a sufficiently small disc, $u(t)\ne0$. Hence
\begin{equation*}
 \nu_1(t)\in C_2
 \quad\Longrightarrow\quad q(\nu_1(t))=0
 \quad\Longrightarrow\quad t=0
 \quad\Longrightarrow\quad \nu_1(t)=a.
\end{equation*}
The required representatives follow by restricting both branches.
\end{proof}

\begin{lemma}[The induced map on a branch]\label{lem:branch-map}
Let $F=(f_1,f_2)$ be a product of nonconstant polynomials. If $C$
is a branch with nonconstant first coordinate, then $F(C)$ determines
a single irreducible branch, denoted $F_*C$. The map induced by $F$
between primitive parameter discs is holomorphic, including at their
centres, and is locally nonconstant.
\end{lemma}
\begin{proof}
Let $\nu$ parametrize $C$. The first coordinate of $F\circ\nu$
is nonconstant. Lemma~\ref{lem:primitive} gives
\begin{equation*}
 F\circ\nu(t)=\nu_*(s(t)^\ell),
 \qquad s'(0)\ne0,\quad \ell\ge1,
\end{equation*}
where $\nu_*$ is primitive. In these two coordinates the induced
map is
\begin{equation}\label{eq:branch-local-degree}
 t\longmapsto s(t)^\ell.
\end{equation}
It is holomorphic and has local degree $\ell$. Replacing either
primitive parametrization changes \eqref{eq:branch-local-degree}
by the biholomorphic transitions in Lemma~\ref{lem:transition}.
Thus both the branch $F_*C$ and the induced holomorphic map are
independent of the chosen parametrizations.
\end{proof}

\subsection{Globalization on a surface of analytic branches}\label{sec:surface}

The following proposition implements Inou's local-to-global construction
\cite[Theorem~2]{Inou} for polynomial maps, and it should be compared
with \cite[Theorem~2.2]{JX}. Its hypotheses contain only a local
functional equation and require no polynomial-like restriction. We
include the construction in order to verify, within a single
realization, both the exact original local conjugacy and a fibre bound
with local multiplicities. The finite-degree framework of Inou is the
source of this step, so we assert here no new globalization principle.

\begin{proposition}[Globalization with a retained graph]\label{prop:globalization}
Let $f_1,f_2\in\C[z]$ be nonconstant polynomials of degrees $d_1,d_2$.
Let $U\subset\C$ be a connected open set, let $h:U\to\C$ be
nonconstant and holomorphic, and suppose there is a nonempty open set $U'\subset U$
such that
\begin{equation}\label{eq:local-functional-equation}
 f_1(U')\subset U,\qquad h\circ f_1=f_2\circ h\quad\text{on }U'.
\end{equation}
Set
\begin{equation*}
 F=(f_1,f_2),\qquad
 \Psi_n(z)=\bigl(\iter{f_1}n(z),\iter{f_2}n(h(z))\bigr),\qquad
 \Gamma=\bigcup_{n\ge0}\Psi_n(U).
\end{equation*}
There exist a connected Hausdorff second-countable Riemann surface $X$,
a nonconstant holomorphic map $G:X\to X$, a holomorphic map
$\pi=(\phi_1,\phi_2):X\to\C^2$, and an open biholomorphic embedding
$j:U\to X$, with the following properties:
\begin{align}
 \pi(X)&=\Gamma,
 &\pi\circ G&=F\circ\pi,\label{eq:global-diagram}\\
 \pi\circ j&=(\id_U,h),
 &G\circ j&=j\circ f_1\quad\text{on }U'\nonumber\\
 \phi_i\circ G&=f_i\circ\phi_i\quad(i=1,2),
 &\phi_1\circ j&=\id_U.\label{eq:global-semiconjugacies}
\end{align}
\begin{equation}\label{eq:pi-locally-injective}
\pi\text{ is injective on a neighbourhood of each point of }X.
\end{equation}
Write $\deg_xG$ for the local degree of $G$ at a point $x\in X$,
namely the order with which $G$ takes the value $G(x)$ near $x$. For
every $y\in X$,
\begin{equation}\label{eq:valence-multiplicity}
 \sum_{x\in G^{-1}(y)}\deg_xG\le d_1d_2.
\end{equation}
In particular, every fibre is finite and has at most $d_1d_2$ points.
No global injectivity, finite valence, or properness of either
projection $\phi_i$ is assumed.
\end{proposition}

\subsubsection{The underlying set and the coordinate discs}
For $n\ge0$ and $z\in U$, let $C(n,z)$ denote the irreducible image
branch of the germ of $\Psi_n$ at $z$, based at $\Psi_n(z)$.
Its existence follows from Lemma~\ref{lem:primitive}, since
$\iter{f_1}n$ is nonconstant. Define
\begin{equation}\label{eq:surface-set}
 X=\{(a,C):a=\Psi_n(z),\ C=C(n,z),\ n\ge0,\ z\in U\}.
\end{equation}
The equality relation is precisely
\begin{equation}\label{eq:equality-of-points}
 (a,C)=(b,D)
 \quad\Longleftrightarrow\quad a=b\text{ and }C=D\text{ as branch germs}.
\end{equation}
Thus two distinct branches over the same point of $\C^2$ represent
different points of $X$.

Fix a representative $(n,z_0)$ of a point of $X$. On a sufficiently
small source neighbourhood $V\ni z_0$, Lemma~\ref{lem:primitive} gives
\begin{equation}\label{eq:chart-factorization}
 \Psi_n(z)=\nu(\eta(z)),\qquad
 \eta(z_0)=0,\qquad \eta(z)=s(z)^\ell,
 \quad s(z_0)=0,\quad s'(z_0)\ne0.
\end{equation}
The open mapping theorem gives
\begin{equation*}
 \disc\varepsilon\subset\eta(V)
\end{equation*}
for some $\varepsilon>0$. Shrink $\varepsilon$ so that $\nu$ is
injective on a slightly larger disc. For $|t|<\varepsilon$, let
$[\nu]_t$ be the image branch of $\nu$ at its parameter $t$ and set
\begin{equation}\label{eq:lifted-chart}
 \widehat\nu:\disc\varepsilon\longrightarrow X,
 \qquad \widehat\nu(t)=(\nu(t),[\nu]_t).
\end{equation}
This map takes its values in $X$. Indeed, for any such $t$, choose
$z\in V$ with $\eta(z)=t$. Since $\eta$ maps every neighbourhood
of $z$ onto a neighbourhood of $t$, \eqref{eq:chart-factorization}
implies
\begin{equation*}
 \bigl(\Psi_n(z),C(n,z)\bigr)
       =\bigl(\nu(t),[\nu]_t\bigr).
\end{equation*}
The injectivity of $\nu$ implies the injectivity of $\widehat\nu$.

\begin{lemma}[Compatibility of the branch charts]\label{lem:atlas}
The maps \eqref{eq:lifted-chart} define a complex atlas on the set
\eqref{eq:surface-set}. Their transition functions are biholomorphic.
\end{lemma}
\begin{proof}
Suppose two lifted charts have a common point,
\begin{equation*}
 \widehat\nu_1(t_1)=\widehat\nu_2(t_2)=x.
\end{equation*}
By \eqref{eq:equality-of-points}, their image branch germs at $x$
coincide. Translating the parameter centres to $t_1,t_2$ and applying
Lemma~\ref{lem:transition}, there are parameter neighbourhoods
$W_i\ni t_i$ and a biholomorphic map $\theta:W_1\to W_2$ with
\begin{equation}\label{eq:chart-transitions}
 \nu_1=\nu_2\circ\theta,
 \qquad \widehat\nu_1=\widehat\nu_2\circ\theta
 \quad\text{on }W_1.
\end{equation}
Thus every point of a chart intersection has an open neighbourhood
in each parameter disc contained in that intersection. The intersections
are open and the transitions are holomorphic.

For a chart-independent definition of the topology, declare
$W\subset X$ open exactly when
\begin{equation*}
 \widehat\nu^{-1}\bigl(W\cap\widehat\nu(\disc\varepsilon)\bigr)
 \text{ is open in }\disc\varepsilon
\end{equation*}
for every chart. Because the transitions in \eqref{eq:chart-transitions}
are biholomorphic, this topology does not depend on the chart. Each
chart image is open by definition, since for any second chart
$\widehat\rho$ the biholomorphic transitions make
$\widehat\rho^{-1}\bigl(\widehat\nu(\disc\varepsilon)\cap\im\widehat\rho\bigr)$
open, and each $\widehat\nu$ is a homeomorphism onto its image. These maps therefore form a complex atlas.
\end{proof}

\begin{lemma}[Separation and second countability]\label{lem:topology}
The topology in Lemma~\ref{lem:atlas} is Hausdorff and second countable.
For every $n\ge0$, the map
\begin{equation*}
 j_n:U\to X,\qquad j_n(z)=(\Psi_n(z),C(n,z))
\end{equation*}
is holomorphic, locally nonconstant, and open. Moreover,
\begin{equation}\label{eq:surface-cover}
 X=\bigcup_{n\ge0}j_n(U).
\end{equation}
\end{lemma}
\begin{proof}
Let $x_1=(a_1,C_1)$ and $x_2=(a_2,C_2)$ be distinct points.
If $a_1\ne a_2$, choose disjoint ambient balls $B_i\ni a_i$.
Restrict the lifted charts so that
\begin{equation*}
 \nu_i(\disc{\varepsilon_i})\subset B_i.
\end{equation*}
Their lifted images are disjoint neighbourhoods of $x_1,x_2$.

If $a_1=a_2=a$, then $C_1\ne C_2$. By
Lemma~\ref{lem:separation}, choose the chart representatives so that
\begin{equation*}
 \nu_1(\disc{\varepsilon_1})\cap
 \nu_2(\disc{\varepsilon_2})\subset\{a\}.
\end{equation*}
A common point of the lifted charts would have to project to $a$.
Their only respective points above $a$ are $x_1$ and $x_2$, by
injectivity, and these are not identified by
\eqref{eq:equality-of-points}. The lifted chart images are again
disjoint. This proves the Hausdorff property.

In the coordinates \eqref{eq:chart-factorization}--\eqref{eq:lifted-chart},
\begin{equation*}
 \widehat\nu^{-1}\circ j_n(z)=\eta(z)=s(z)^\ell.
\end{equation*}
Thus $j_n$ is holomorphic and locally nonconstant. Its openness follows
from the open mapping theorem in these coordinates. Equation
\eqref{eq:surface-cover} follows directly from \eqref{eq:surface-set}.

Let $\mathcal V$ be a countable base for the open set $U\subset\C$.
The family
\begin{equation}\label{eq:countable-base}
 \mathcal W=\{j_n(V):n\ge0,\ V\in\mathcal V\}
\end{equation}
is countable and consists of open sets. If $x\in W$ with $W\subset X$
open, write $x=j_n(z)$ and choose
\begin{equation*}
 z\in V\in\mathcal V,\qquad V\subset j_n^{-1}(W).
\end{equation*}
Then
\begin{equation*}
 x\in j_n(V)\subset W.
\end{equation*}
Hence \eqref{eq:countable-base} is a base for $X$.
\end{proof}

\subsubsection{The lifted dynamics and the original local germ}

\begin{lemma}[Projection, connectedness, and local tracking]\label{lem:tracking}
Define
\begin{equation}\label{eq:projection-lift-definition}
 \pi(a,C)=a,\qquad
 G(a,C)=(F(a),F_*C).
\end{equation}
These are well-defined holomorphic maps and satisfy
\begin{equation}\label{eq:jn-dynamics}
 G\circ j_n=j_{n+1},\qquad \pi\circ G=F\circ\pi.
\end{equation}
The map $\pi$ is locally injective. The surface $X$ is connected, and
$j=j_0$ is an open biholomorphic embedding with
\begin{equation}\label{eq:tracking-exact}
 \phi_1\circ j=\id_U,\qquad\phi_2\circ j=h,
 \qquad G\circ j=j\circ f_1\text{ on }U'.
\end{equation}
\end{lemma}
\begin{proof}
By Lemma~\ref{lem:branch-map}, the germ $F_*C$ depends only on $C$.
For $C=C(n,z)$ it equals $C(n+1,z)$, so
\eqref{eq:projection-lift-definition} takes its values in $X$ and
is independent of the representative $(n,z)$. In a lifted chart,
\begin{equation*}
 \pi\circ\widehat\nu=\nu.
\end{equation*}
This proves that $\pi$ is holomorphic and locally injective. With
primitive coordinates in the source and the target, the expression
for $G$ is the holomorphic map \eqref{eq:branch-local-degree}.
Hence $G$ is holomorphic, also at centres projecting to singular
branches. The two identities in \eqref{eq:jn-dynamics} follow from
$F\circ\Psi_n=\Psi_{n+1}$.

Equation \eqref{eq:local-functional-equation} gives, for $z\in U'$,
\begin{align*}
 \Psi_{n+1}(z)
 &=\bigl(\iter{f_1}{n+1}(z),\iter{f_2}{n+1}(h(z))\bigr)\\
 &=\bigl(\iter{f_1}n(f_1(z)),\iter{f_2}n(h(f_1(z)))\bigr)
 =\Psi_n(f_1(z)).\nonumber
\end{align*}
This equality holds on an open neighbourhood of each $z\in U'$.
Since $f_1$ is locally nonconstant and open, it also identifies the
image branch germs. Therefore
\begin{equation}\label{eq:consecutive-jn}
 j_{n+1}(z)=j_n(f_1(z))\qquad(z\in U').
\end{equation}
In particular,
\begin{equation}\label{eq:consecutive-overlap}
 \varnothing\ne j_{n+1}(U')=j_n(f_1(U'))
 \subset j_{n+1}(U)\cap j_n(U).
\end{equation}
Each $j_n(U)$ is connected and open. Their consecutive nonempty
intersections, together with \eqref{eq:surface-cover}, imply that
$X$ is connected.

The first coordinate of $\Psi_0$ is $z$, so $j=j_0$ is injective.
It is open by Lemma~\ref{lem:topology}, and
\begin{equation*}
 j^{-1}=\phi_1|_{j(U)}.
\end{equation*}
This proves that $j$ is biholomorphic onto its open image. The first
two identities in \eqref{eq:tracking-exact} follow from
$\pi j=\Psi_0$. Taking $n=0$ in \eqref{eq:consecutive-jn} gives
the last identity. Finally, $G$ is nonconstant because
$G j=j_1$ and $j_1$ is locally nonconstant.
\end{proof}

\subsubsection{The finite-fibre estimate}

\begin{lemma}[Valence, including local multiplicities]\label{lem:fibre}
For the maps in Lemma~\ref{lem:tracking}, every $y\in X$ satisfies
\eqref{eq:valence-multiplicity}.
\end{lemma}
\begin{proof}
Set $D=d_1d_2$. For each $a=(a_1,a_2)\in\C^2$,
\begin{equation}\label{eq:ambient-valence}
 F^{-1}(a)=f_1^{-1}(a_1)\times f_2^{-1}(a_2),
 \qquad \#F^{-1}(a)\le D.
\end{equation}
Choose any finite collection of distinct points
\begin{equation}\label{eq:finite-subfibre}
 x_1,\ldots,x_s\in G^{-1}(y),\qquad e_i=\deg_{x_i}G\ge1.
\end{equation}
We prove $\sum_{i=1}^s e_i\le D$.

Write $a_i=\pi(x_i)$. Choose disjoint source neighbourhoods
$V_i\ni x_i$ so small that $\pi|_{V_i}$ is injective. Shrink them
further to arrange
\begin{equation}\label{eq:pairwise-projection-separation}
 \pi(V_i)\cap\pi(V_j)=\varnothing\quad(a_i\ne a_j),
 \qquad
 \pi(V_i)\cap\pi(V_j)\subset\{a_i\}\quad(a_i=a_j,
 i\ne j).
\end{equation}
For unequal base points this follows from disjoint ambient balls.
For equal base points the branch germs are distinct by
\eqref{eq:equality-of-points}, and Lemma~\ref{lem:separation} applies.
There are only finitely many pairs, so all these restrictions can
be imposed simultaneously.

Take a target coordinate $w$ with $w(y)=0$. At $x_i$, a source
coordinate $t_i$ may be chosen so that
\begin{equation}\label{eq:local-fibre-normal-form}
 w\circ G=t_i^{e_i},\qquad t_i(x_i)=0.
\end{equation}
Indeed, in an arbitrary source coordinate $u$,
$w\circ G=u^{e_i}v(u)$ with $v(0)\ne0$, and the coordinate
$t_i=u\,v(u)^{1/e_i}$ gives \eqref{eq:local-fibre-normal-form}.
Choose $r_i>0$ with $t_i^{-1}(\disc{r_i})\subset V_i$.
For a sufficiently small nonzero $\tau$ satisfying
\begin{equation*}
 0<|\tau|<\min_{1\le i\le s}r_i^{e_i},
\end{equation*}
put $y_\tau=w^{-1}(\tau)$. Each equation
$t_i^{e_i}=\tau$ has exactly $e_i$ distinct solutions in
$\disc{r_i}$. Thus there are points
\begin{equation}\label{eq:regular-source-points}
 u_{i,1},\ldots,u_{i,e_i}\in V_i\setminus\{x_i\},
 \qquad G(u_{i,j})=y_\tau.
\end{equation}
Their projections are all distinct. Within one $V_i$ this follows
from the injectivity of $\pi|_{V_i}$. For different $V_i,V_k$ with
equal base point, a coincidence of projections would, by
\eqref{eq:pairwise-projection-separation}, force
\begin{equation*}
 \pi(u_{i,j})=a_i=\pi(x_i)
 \quad\Longrightarrow\quad u_{i,j}=x_i,
\end{equation*}
contrary to \eqref{eq:regular-source-points}. Distinct base points
are already separated in \eqref{eq:pairwise-projection-separation}.

On the other hand,
\begin{equation*}
 F(\pi(u_{i,j}))=\pi(G(u_{i,j}))=\pi(y_\tau).
\end{equation*}
We have produced $\sum_i e_i$ distinct points in the same fibre of
$F$. Equation \eqref{eq:ambient-valence} gives
\begin{equation}\label{eq:finite-multiplicity-bound}
 \sum_{i=1}^s e_i\le D.
\end{equation}
If $G^{-1}(y)$ were infinite, a choice of $D+1$ distinct points in
\eqref{eq:finite-subfibre} would contradict
\eqref{eq:finite-multiplicity-bound}. Thus the fibre is finite.
Taking all its points proves \eqref{eq:valence-multiplicity}.
\end{proof}

\begin{proof}[Completion of Proposition~\ref{prop:globalization}]
Lemmas~\ref{lem:atlas} and \ref{lem:topology} establish the complex
manifold structure and the separation and countability axioms.
Lemma~\ref{lem:tracking} gives connectedness and all identities in
\eqref{eq:global-diagram}--\eqref{eq:global-semiconjugacies}.
Surjectivity onto $\Gamma$ is the equality
\begin{equation*}
 \pi(X)=\bigcup_{n\ge0}\pi(j_n(U))
       =\bigcup_{n\ge0}\Psi_n(U)=\Gamma.
\end{equation*}
Lemma~\ref{lem:fibre} supplies the remaining fibre bound.
\end{proof}

\subsubsection{The retained neutral fixed point}

\begin{corollary}[Marked globalization of a commuting germ]\label{cor:marked}
Let $f(z)=\lambda z+z^2$, $\lambda\ne0$, and $h\in\Cent(f)$.
There exist $X,G,\phi_1,\phi_2,j$ as above and a point $x_*=j(0)$
such that
\begin{align}
 G(x_*)&=x_*,&G'(x_*)&=\lambda,\label{eq:marked-multiplier}\\
 \phi_i\circ G&=f\circ\phi_i,&\phi_i(x_*)&=0,
 &\mathrm d\phi_i(x_*)&\ne0\quad(i=1,2),\label{eq:marked-projections}\\
 \sum_{G(x)=y}\deg_xG&\le4&&\text{for every }y\in X.\label{eq:marked-valence}
\end{align}
The germ of $G$ at $x_*$ is analytically conjugate to $f$, and
\begin{equation}\label{eq:local-h-reconstruction}
 h=\phi_2\circ(\phi_1|_{j(U)})^{-1}
\end{equation}
as a germ at $0$.
\end{corollary}
\begin{proof}
Choose discs $0\in U'\subset U$ on which $h$ is holomorphic,
$h'(0)\ne0$, $f(U')\subset U$, and $hf=fh$ on $U'$.
Apply Proposition~\ref{prop:globalization} with $f_1=f_2=f$.
The exact identity
\begin{equation}\label{eq:local-fixed-conjugacy}
 G\circ j=j\circ f\quad\text{on }U'
\end{equation}
gives
\begin{equation*}
 G(j(0))=j(0),\qquad
 G'(x_*)j'(0)=j'(0)\lambda,
 \qquad j'(0)\ne0.
\end{equation*}
Derivatives here are computed in any local coordinate at $x_*$.
Thus \eqref{eq:marked-multiplier} holds and $j$ conjugates the
marked germs. Furthermore,
\begin{equation}\label{eq:projection-derivatives}
 (\phi_1\circ j)'(0)=1,
 \qquad(\phi_2\circ j)'(0)=h'(0)\ne0.
\end{equation}
This proves local invertibility of both projections and
\eqref{eq:local-h-reconstruction}. Equation \eqref{eq:marked-valence}
is \eqref{eq:valence-multiplicity} with $d_1=d_2=2$.
\end{proof}

\begin{remark}\label{rem:no-global-assumptions}
The proof uses neither a nested family $\Psi_n(U)$ nor a closed
analytic structure on $\Gamma$ in the ambient topology. Connectedness
comes from \eqref{eq:consecutive-overlap}, and the complex structure
comes from primitive branch coordinates. At this stage only $G$ has
finite valence, while the projections $\phi_i$ are not yet known to
have finite valence or to be proper.
\end{remark}

\begin{remark}[Negative iterates and the projection pair]\label{rem:negative-iterates}
For the germ $h=f^{-1}$, the canonical globalized curve is
\begin{equation*}
 \Gamma=\{(f(w),w):w\in\C\},
\end{equation*}
with model
\begin{equation*}
 X=\C,\qquad G=f,\qquad\phi_1=f,\qquad\phi_2=\id.
\end{equation*}
The first projection is not globally injective, while the pair
$(\phi_1,\phi_2)$ is globally injective. At $0$, the inverse branch
of $\phi_1$ fixing $0$ recovers $h=f^{-1}$.
More generally, for $\ell\ge1$ the model
\begin{equation*}
 G=f,\qquad\phi_1=f^{\circ\ell},\qquad\phi_2=\id
\end{equation*}
satisfies the semiconjugacy equations and reconstructs
$f^{\circ(-\ell)}$. Thus the use of a locally invertible projection,
rather than a globally invertible function on the original plane,
is necessary for the negative iterates in the centraliser.
\end{remark}

\section{Uniformization: the lifted surface is the plane}\label{sec:unif-plane}

\subsection{Uniformization at a nonlinearizable fixed point}\label{sec:uniformization}

A connected Riemann surface is called hyperbolic if its universal
cover is the unit disc. The following elementary consequence of
Schwarz's lemma is the fixed-point rigidity on which the proof
rests, and the derivative of a self-map at a fixed point is
independent of the chosen local coordinate. Covering-space and
uniformization facts are used in the sense of \cite{Forster},
whereas the Schwarz and Picard arguments are classical complex
analysis~\cite{Ahlfors}. The local conjugacy in
Lemma~\ref{lem:schwarz} is proved explicitly, since it is the link
between the neutral germ and the global surface.

\begin{lemma}[Schwarz equality on a hyperbolic surface]\label{lem:schwarz}
Let $Y$ be a connected hyperbolic Riemann surface and $T:Y\to Y$
a holomorphic map. Suppose
\begin{equation*}
 T(y_*)=y_*,\qquad T'(y_*)=\lambda,\qquad |\lambda|=1.
\end{equation*}
Then the germ of $T$ at $y_*$ is analytically conjugate to $w\mapsto\lambda w$.
\end{lemma}
\begin{proof}
Write $\D=\{z\in\C:|z|<1\}$ for the unit disc. Take a universal
covering
\begin{equation*}
 p:\D\longrightarrow Y,\qquad p(0)=y_*.
\end{equation*}
Since $\D$ is simply connected, $T\circ p$ has a lift
$\widetilde T:\D\to\D$ with $\widetilde T(0)=0$. The lift is
holomorphic because $p$ is locally biholomorphic, and
\begin{equation}\label{eq:covering-lift-equation}
 p\circ\widetilde T=T\circ p.
\end{equation}
Differentiating in a coordinate at $y_*$ gives
\begin{equation*}
 p'(0)\widetilde T'(0)=T'(y_*)p'(0),
 \qquad p'(0)\ne0,
 \qquad \widetilde T'(0)=\lambda.
\end{equation*}
Schwarz's lemma implies $|\widetilde T(w)|\le|w|$. Therefore
\begin{equation*}
 A(w)=\begin{cases}\widetilde T(w)/w,&w\ne0,\\
                       \widetilde T'(0),&w=0,
       \end{cases}
 \qquad |A(w)|\le1,\quad |A(0)|=1.
\end{equation*}
The maximum modulus principle yields $A\equiv\lambda$, hence
\begin{equation}\label{eq:linear-lift}
 \widetilde T(w)=\lambda w\quad(w\in\D).
\end{equation}
Choose $r>0$ so small that $p|_{\disc r}$ is biholomorphic onto its
image. By \eqref{eq:covering-lift-equation} and
\eqref{eq:linear-lift},
\begin{equation*}
 T(p(w))=p(\lambda w)\qquad(|w|<r).
\end{equation*}
Since $|\lambda|=1$, both sides are defined on the same disc.
The local inverse of $p$ is the required linearizing coordinate.
\end{proof}

\begin{lemma}[Finite valence on $\C$ and $\C^*$]\label{lem:finite-valence}
A nonconstant entire function whose fibres have uniformly bounded
cardinality is a polynomial. If a nonconstant holomorphic map
$T:\C^*\to\C^*$ has uniformly bounded fibre cardinality, then
\begin{equation}\label{eq:punctured-monomial}
 T(z)=az^m,\qquad a\in\C^*,\quad m\in\Z\setminus\{0\}.
\end{equation}
At every fixed point $z_*\in\C^*$ of this map,
\begin{equation*}
 T'(z_*)=m.
\end{equation*}
\end{lemma}
\begin{proof}
Let $T$ first be entire. If its singularity at infinity were essential,
the great Picard theorem would give a finite value $b$ that is taken
infinitely often in every neighbourhood of infinity. Thus
\begin{equation*}
 \#T^{-1}(b)=\infty,
\end{equation*}
which contradicts finite valence, so the singularity at infinity is
removable or a pole. Removability would make $T$ bounded outside a
large disc, hence bounded on $\C$ and constant by Liouville's theorem.
Infinity is therefore a pole, and $T$ is a polynomial.

For $T:\C^*\to\C^*$, an essential singularity at $0$ or at infinity
would, again by the great Picard theorem, contradict finite valence.
Each isolated singularity is therefore removable or a pole, so $T$
extends to a rational function on the Riemann sphere
$\Chat=\C\cup\{\infty\}$. Since $T$ has no zeros or
poles on $\C^*$, all of its zeros and poles lie in $\{0,\infty\}$.
Let $m=\ord_0T$. Then $z^{-m}T(z)$ has no zeros or poles anywhere
on the sphere, so it is a nonzero constant. This proves
\eqref{eq:punctured-monomial}, in which the case $m=0$ is excluded
because it would make $T$ constant. If $az_*^m=z_*$, then
\begin{equation*}
 T'(z_*)=ma z_*^{m-1}
         =m\frac{az_*^m}{z_*}=m.
\end{equation*}
\end{proof}

\begin{proposition}[The lifted surface is the plane]\label{prop:plane}
Assume that $f(z)=\lambda z+z^2$ is nonlinearizable at $0$, where
$\lambda=e^{2\pi i\alpha}$ and $\alpha$ is irrational. For the
marked globalization in Corollary~\ref{cor:marked},
\begin{equation}\label{eq:plane-conclusion}
 X\simeq\C.
\end{equation}
In a global coordinate on $X$, both $\phi_i$ are entire and $G$ is
a polynomial satisfying
\begin{equation}\label{eq:preliminary-degree}
 2\le\deg G\le4.
\end{equation}
\end{proposition}
\begin{proof}
If $X$ were hyperbolic, Lemma~\ref{lem:schwarz} applied at $x_*$
would linearize $G$, and composing that local coordinate with $j$ as
in \eqref{eq:local-fixed-conjugacy} would linearize $f$, which the
hypothesis excludes. Hence $X$ is nonhyperbolic.

By uniformization and covering-space theory~\cite{Forster}, the
remaining possibilities are
\begin{equation}\label{eq:surface-options}
 X\simeq\Chat,\qquad \C/\Lambda,\qquad \C,\qquad\C^*,
\end{equation}
where $\Lambda\subset\C$ is a lattice of rank two. To make the
quotient alternatives explicit, suppose first that the universal
cover is $\C$. Its deck transformations are affine maps
$z\mapsto az+b$, and a nonidentity deck transformation has no
fixed point. Thus
\begin{equation*}
 a\ne1\ \Longrightarrow\ z=\frac{b}{1-a}\text{ is fixed}
 \quad\Longrightarrow\quad a=1.
\end{equation*}
The deck group is therefore a discrete additive subgroup of $\C$ of
rank $0$, $1$, or $2$. These ranks give respectively $\C$, $\C^*$
(via an exponential coordinate), or $\C/\Lambda$. If instead the
universal cover is the sphere, then every nonidentity M\"obius map
has a fixed point, so the deck group is trivial and $X\simeq\Chat$.
This gives exactly \eqref{eq:surface-options}.

The compact alternatives $\Chat$ and $\C/\Lambda$ are impossible.
Indeed, $\phi_1:X\to\C$ is a nonconstant holomorphic function by
\eqref{eq:projection-derivatives}, and on a compact surface the
maximum modulus principle would force
\begin{equation*}
 \phi_1\equiv\text{constant},
\end{equation*}
a contradiction. It is essential here that $\phi_1$ takes values in
$\C$ rather than merely in the sphere.

Suppose next that $X\simeq\C^*$. The fibre bound
\eqref{eq:marked-valence} together with Lemma~\ref{lem:finite-valence}
then gives $G(z)=az^m$, and the retained fixed point gives
\begin{equation*}
 \lambda=G'(x_*)=m\in\Z\setminus\{0\}.
\end{equation*}
Since $|\lambda|=1$, the multiplier $m$ can only be $1$ or $-1$,
both of which are roots of unity, whereas
$\lambda=e^{2\pi i\alpha}$ with $\alpha$ irrational is not. Thus
$X\not\simeq\C^*$, and \eqref{eq:plane-conclusion} follows.

In the plane coordinate, $G$ and $\phi_i$ are entire. The map $G$
is nonconstant and has finite valence, so it is polynomial by
Lemma~\ref{lem:finite-valence}. If $d=\deg G$, then the fundamental
theorem of algebra, counted with local multiplicities, gives
\begin{equation*}
 d=\sum_{G(x)=y}\deg_xG\le4.
\end{equation*}
If $d=1$, then the marked fixed point and its multiplier force
\begin{equation*}
 G(w)=x_*+\lambda(w-x_*).
\end{equation*}
The coordinate $z\mapsto j(z)-x_*$ would then linearize $f$ by
\eqref{eq:local-fixed-conjugacy}, again a contradiction. Therefore
$d\ge2$, which proves \eqref{eq:preliminary-degree}.
\end{proof}

\begin{remark}[The arithmetic input is qualitative]\label{rem:arithmetic-qualitative}
The proof uses the Brjuno condition only through \eqref{eq:BY}, so
it places no lower bound on the individual continued-fraction
contributions. The implication in Theorem~\ref{thm:main} therefore
applies equally when
\begin{equation*}
 \sum_{n\ge0}\frac{\log q_{n+1}}{q_n}=\infty,
 \qquad \frac{\log q_{n+1}}{q_n}\longrightarrow0.
\end{equation*}
No step requires a special subsequence of exceptionally close
rational approximations. The arithmetic statement is obtained by
combining a qualitative nonlinearizability criterion with the
analytic-geometric argument.
\end{remark}

\begin{remark}[The role of nonlinearizability]\label{rem:nonlinearizability}
At a linearizable neutral fixed point, the conclusion of
Lemma~\ref{lem:schwarz} is compatible with the germ and does not
exclude a hyperbolic surface, and a plane lift of degree one would
not contradict linearizability either. The two exclusions
\begin{equation*}
 X\text{ hyperbolic}\Longrightarrow f\text{ linearizable},
 \qquad
 \deg G=1\Longrightarrow f\text{ linearizable}
\end{equation*}
are precisely where the hypothesis in Theorem~\ref{thm:main}
is used. Once these possibilities are excluded, the two quadratic
semiconjugacy equations force polynomial projections, which
separates the geometric obstruction from the established algebraic
classification of global polynomial commutants.
\end{remark}

\section{Polynomial rigidity of the commuting graph}\label{sec:rigidity}

\subsection{A growth estimate for the entire projections}\label{sec:growth}

The conclusion that $G$ is polynomial does not by itself imply
that the entire functions $\phi_i$ are polynomial. The general
obstruction to transcendental semiconjugacies, due to
Eremenko~\cite[Theorem~2]{Eremenko}, is used for projection rigidity
in Inou~\cite[Section~3.2]{Inou} and is restated in
Ji--Xie~\cite[Theorem~2.3]{JX}. Rather than importing its general
orbifold classification, we prove the precise quadratic
specialization needed here by Rouch\'e's theorem and Cauchy's
estimates~\cite{Ahlfors}, and no finite-valence hypothesis on
$\phi_i$ is required.

\begin{lemma}[Entire semiconjugacies to a quadratic polynomial]\label{lem:growth}
Let $G\in\C[z]$ have degree $d\ge2$ and let $\phi$ be a
nonconstant entire function. Suppose $|\lambda|=1$ and
\begin{equation}\label{eq:entire-semiconjugacy}
 \phi(G(z))=\phi(z)^2+\lambda\phi(z)\qquad(z\in\C).
\end{equation}
Then $\phi$ is a polynomial and $d=2$.
\end{lemma}
\begin{proof}
Write
\begin{equation*}
 G(z)=a_dz^d+\sum_{j=0}^{d-1}a_jz^j,\qquad a_d\ne0.
\end{equation*}
Choose
\begin{equation*}
 0<c\le\min\{1,|a_d|/4\},\qquad R\ge1,
 \qquad \sum_{j=0}^{d-1}|a_j|r^j\le\frac{|a_d|}{4}r^d
 \quad(r\ge R).
\end{equation*}
For $|w|\le cr^d$ and $|z|=r$, we have
\begin{align*}
 |G(z)-w-a_dz^d|
 &\le\sum_{j=0}^{d-1}|a_j|r^j+|w|\\
 &\le\frac{|a_d|}{4}r^d+cr^d
 \le\frac{|a_d|}{2}r^d<|a_dz^d|.\nonumber
\end{align*}
Rouch\'e's theorem implies that $G(z)-w$ has $d$ zeros in
$\disc r$, counted with multiplicities. Therefore
\begin{equation}\label{eq:large-disc-image}
 \overline{\disc{cr^d}}\subset G(\disc r),
 \qquad \overline{\disc{cr^2}}\subset G(\disc r)
 \quad(r\ge R),
\end{equation}
where the second inclusion uses $r\ge1$ and $d\ge2$.

Define
\begin{equation*}
 M(r)=\max_{|z|\le r}|\phi(z)|,
 \qquad A(r)=\max\{1,M(r)\},\qquad H(r)=2A(r).
\end{equation*}
Given $|w|\le cr^2$, choose $z\in\disc r$ with $G(z)=w$,
using \eqref{eq:large-disc-image}. Then
\begin{equation*}
 |\phi(w)|=|\phi(z)^2+\lambda\phi(z)|
 \le M(r)^2+M(r)\le2A(r)^2.
\end{equation*}
Since $2A(r)^2\ge1$, this yields
\begin{equation*}
 H(cr^2)=2\max\{1,M(cr^2)\}
 \le4A(r)^2=H(r)^2\qquad(r\ge R).
\end{equation*}

Choose $r_0\ge R$ with $t=cr_0>1$ and set
\begin{equation*}
 r_{n+1}=cr_n^2\qquad(n\ge0).
\end{equation*}
By induction,
\begin{equation}\label{eq:radii-explicit}
 r_n=c^{-1}t^{2^n},\qquad r_n\nearrow\infty,
 \qquad \log H(r_n)\le2^n\log H(r_0).
\end{equation}
For any $r\ge r_0$, choose $n$ with $r_n\le r<r_{n+1}$.
Monotonicity of $H$ and \eqref{eq:radii-explicit} give
\begin{align*}
 \log H(r)&\le\log H(r_{n+1})
            \le2^{n+1}\log H(r_0),\\
 \log(cr)&\ge\log(cr_n)=2^n\log t.\nonumber
\end{align*}
With
\begin{equation*}
 K=\frac{2\log H(r_0)}{\log t}>0,
\end{equation*}
we obtain for every $r\ge r_0$,
\begin{equation}\label{eq:power-bound}
 M(r)\le H(r)\le\exp(K\log(cr))=(cr)^K.
\end{equation}
In particular, the estimate holds for every sufficiently large
radius, not just the sequence $r_n$.

Write $\phi(z)=\sum_{j\ge0}b_jz^j$. For each integer $j>K$,
Cauchy's coefficient estimate and \eqref{eq:power-bound} imply
\begin{equation*}
 |b_j|\le\frac{M(r)}{r^j}\le c^K r^{K-j}
 \xrightarrow[r\to\infty]{}0.
\end{equation*}
Hence $b_j=0$ for all integers $j>K$, and $\phi$ is polynomial.
If $m=\deg\phi\ge1$, comparison of degrees in
\eqref{eq:entire-semiconjugacy} gives
\begin{equation*}
 md=\deg(\phi\circ G)
    =\deg(\phi^2+\lambda\phi)=2m.
\end{equation*}
Thus $d=2$.
\end{proof}

\begin{remark}
The hypothesis $|\lambda|=1$ only fixes the constant in the estimate
$|\phi(z)^2+\lambda\phi(z)|\le(1+|\lambda|)A(r)^2$, so the same
iteration proves polynomial growth for any complex multiplier
$\lambda$.
\end{remark}

\begin{proposition}[Polynomial realization of the commuting graph]\label{prop:polynomial-realization}
Under the assumptions of Theorem~\ref{thm:main}, the marked
globalization admits a plane coordinate in which
\begin{equation}\label{eq:polynomial-realization}
 G,\phi_1,\phi_2\in\C[z],\qquad
 \deg G=2,\qquad\deg\phi_i\ge1.
\end{equation}
The identities and the nonzero derivatives in
\eqref{eq:marked-multiplier}--\eqref{eq:marked-projections} remain valid.
\end{proposition}
\begin{proof}
Proposition~\ref{prop:plane} gives a plane coordinate, an entire pair
$\phi_1,\phi_2$, and $2\le\deg G\le4$. The projection identities
are
\begin{equation*}
 \phi_i\circ G=\phi_i^2+\lambda\phi_i\qquad(i=1,2).
\end{equation*}
Neither projection is constant, by
\eqref{eq:projection-derivatives}. Applying Lemma~\ref{lem:growth}
to both functions proves \eqref{eq:polynomial-realization}.
Coordinate changes on the surface do not alter the marked
functional equations or the nonvanishing of the differentials.
\end{proof}

\subsection{The polynomial commutant of the quadratic map}\label{sec:polynomial}

For a polynomial of degree $d\ge2$, ``special'' means affinely
conjugate to $z^d$ or to a signed Chebyshev polynomial $\pm T_d$.
We record below the precise global factorization theorem that is
used, which is an established result about polynomials rather than
an assertion about local germs.

\begin{theorem}[Pakovich's minimal-degree factorization]\label{thm:pakovich}
Let $A,B\in\C[z]$ be non-special polynomials of degree at least two,
and define
\begin{equation}\label{eq:E-AB}
 \mathcal E(A,B)=\{X\in\C[z]:\deg X\ge2,\quad A\circ X=X\circ B\}.
\end{equation}
Suppose $\mathcal E(A,B)\ne\varnothing$, and choose
$X_0\in\mathcal E(A,B)$ of minimal degree. Then every
$X\in\mathcal E(A,B)$ has a factorization
\begin{equation*}
 X=\widetilde A\circ X_0,
 \qquad\widetilde A\circ A=A\circ\widetilde A,
\end{equation*}
with a nonconstant polynomial $\widetilde A$, whose degree may be one.
\end{theorem}
This is the forward implication of
\cite[Theorem~1.3, p.~1421]{Pakovich}. The exclusion of degree-one
maps in \eqref{eq:E-AB} is part of the statement, and it ensures
that, when $A=B=f$ is quadratic, $X_0=f$ is a minimal-degree
choice. The global commutation problem originates in the work of
Ritt \cite{Ritt}, and the formulation above also specifies the
right factor needed here.

\begin{lemma}[No nonidentity affine commutants]\label{lem:affine}
For $f(z)=z^2+\lambda z$, the only affine polynomial commuting with
$f$ is the identity.
\end{lemma}
\begin{proof}
Let $L(z)=az+b$ with $a\ne0$. Expanding $L\circ f=f\circ L$ gives
\begin{equation*}
 az^2+a\lambda z+b
 =a^2z^2+(2ab+\lambda a)z+b^2+\lambda b.
\end{equation*}
The coefficients of $z^2$ and $z$ imply respectively
\begin{equation*}
 a=a^2\Longrightarrow a=1,
 \qquad a\lambda=2ab+\lambda a\Longrightarrow 2ab=0
 \Longrightarrow b=0.
\end{equation*}
Thus $L=\id$.
\end{proof}

\begin{lemma}[The quadratic map is non-special]\label{lem:nonspecial}
If $\lambda=e^{2\pi i\alpha}$ with $\alpha$ irrational, then
$f(z)=z^2+\lambda z$ is non-special.
\end{lemma}
\begin{proof}
An affine conjugacy preserves the multiplier of every finite fixed
point. The finite fixed-point multipliers of the special quadratic
models are
\begin{align*}
 z^2&:\quad 0,\ 2,\\
 T_2(z)=2z^2-1&:\quad 4,\ -2,\nonumber\\
 -T_2(z)=-2z^2+1&:\quad -2,\ 4.\nonumber
\end{align*}
For example, the fixed points of $T_2$ are $1$ and $-1/2$, and
$T_2'(z)=4z$, so the other values follow by the same substitution.
The multiplier of $f$ at $0$ is $\lambda$, whereas
\begin{equation*}
 |\lambda|=1,\quad \lambda\text{ not a root of unity}
 \quad\Longrightarrow\quad \lambda\notin\{0,2,4,-2\}.
\end{equation*}
Therefore $f$ is not affinely conjugate to any special quadratic.
\end{proof}

\begin{proposition}[Polynomial commutants are iterates]\label{prop:polynomial-centre}
Let $\lambda=e^{2\pi i\alpha}$ with $\alpha$ irrational and
$f(z)=\lambda z+z^2$. Every nonconstant polynomial $P$ satisfying
$P\circ f=f\circ P$ is $\iter f n$ for some $n\ge0$.
\end{proposition}
\begin{proof}
If $\deg P=1$, Lemma~\ref{lem:affine} gives $P=\id=\iter f0$.
Suppose $\deg P\ge2$. By Lemma~\ref{lem:nonspecial}, $f$ is
non-special. The set $\mathcal E(f,f)$ contains $f$, and
\begin{equation}\label{eq:minimal-degree-two}
 2\le\min_{X\in\mathcal E(f,f)}\deg X\le\deg f=2.
\end{equation}
Thus Theorem~\ref{thm:pakovich} applies with $A=B=X_0=f$ and gives
\begin{equation*}
 P=P_1\circ f,\qquad P_1\circ f=f\circ P_1,
 \qquad\deg P_1=\frac{\deg P}{2}.
\end{equation*}
If $\deg P_1\ge2$, we apply the same theorem again. At every step
where the remaining degree is at least two, this gives
\begin{equation*}
 P_j=P_{j+1}\circ f,\qquad P_{j+1}f=fP_{j+1},
 \qquad \deg P_{j+1}=\tfrac12\deg P_j.
\end{equation*}
A strictly decreasing sequence of positive integer degrees cannot
continue indefinitely, so for some $n\ge0$ the process ends with
$\deg P_n=1$, and
\begin{equation*}
 P=P_n\circ\iter f n.
\end{equation*}
Lemma~\ref{lem:affine} then gives $P_n=\id$, which proves
$P=\iter f n$.
\end{proof}

\section{The main theorem and its consequences}\label{sec:conclude}

\subsection{Proof of the main theorem and the global graph}\label{sec:mainproof}

\subsubsection{Affine normalization and reconstruction of the germ}

\begin{proof}[Proof of Theorem~\ref{thm:main}]
Fix any $h\in\Cent(f)$. Choose representatives near $0$ and apply
Corollary~\ref{cor:marked}. The data satisfy
\begin{equation*}
 Gj=jf,\qquad \phi_iG=f\phi_i,\qquad
 \phi_1j=\id,\qquad\phi_2j=h,
\end{equation*}
with a retained interior fixed point $x_*=j(0)$.
Proposition~\ref{prop:plane} and
Proposition~\ref{prop:polynomial-realization} identify $X$ with
$\C$ and give
\begin{equation*}
 G,\phi_1,\phi_2\in\C[w],\qquad\deg G=2.
\end{equation*}
Since $G(x_*)=x_*$ and $G'(x_*)=\lambda$, its exact quadratic
Taylor expansion is
\begin{equation}\label{eq:quadratic-G-normal-form}
 G(w)=x_*+\lambda(w-x_*)+a(w-x_*)^2,
 \qquad a\ne0.
\end{equation}
Having fixed the plane coordinate $w$ on the identified surface
$X\simeq\C$, we make an affine change of this coordinate. Define the
affine map
\begin{equation*}
 A(w)=a(w-x_*),\qquad A^{-1}(z)=x_*+z/a.
\end{equation*}
Direct substitution into \eqref{eq:quadratic-G-normal-form} gives
\begin{align*}
 A(G(A^{-1}(z)))
 &=a\left(\lambda\frac za+a\frac{z^2}{a^2}\right)\\
 &=\lambda z+z^2=f(z).\nonumber
\end{align*}
Thus
\begin{equation}\label{eq:A-conjugacy}
 A\circ G\circ A^{-1}=f.
\end{equation}

For $i=1,2$, put
\begin{equation*}
 P_i=\phi_i\circ A^{-1}\in\C[z].
\end{equation*}
These polynomials are nonconstant, fix $0$, and have nonzero
first derivatives there. Using \eqref{eq:A-conjugacy} and the
semiconjugacy equations,
\begin{align*}
 P_i\circ f
 &=\phi_i\circ A^{-1}\circ A\circ G\circ A^{-1}\\
 &=\phi_i\circ G\circ A^{-1}
  =f\circ\phi_i\circ A^{-1}
  =f\circ P_i.\nonumber
\end{align*}
Proposition~\ref{prop:polynomial-centre} now applies and yields
\begin{equation}\label{eq:P-i-iterates}
 P_1=\iter f{n_1},\qquad P_2=\iter f{n_2},
 \qquad n_1,n_2\in\Z_{\ge0}.
\end{equation}

We finally specify the inverse branch in the reconstruction.
Since
\begin{equation*}
 P_1(0)=0,\qquad P_1'(0)=\lambda^{n_1}\ne0,
\end{equation*}
the unique inverse germ of $P_1$ fixing $0$ is
$\iter f{-n_1}$. Moreover,
\begin{equation*}
 P_1^{-1}=A\circ\phi_1^{-1}
\end{equation*}
when the inverse of $\phi_1$ is taken at $x_*$. Consequently
\begin{align*}
 h
 &=\phi_2\circ\phi_1^{-1}
  =(\phi_2\circ A^{-1})\circ(A\circ\phi_1^{-1})\\
 &=P_2\circ P_1^{-1}
  =\iter f{n_2}\circ\iter f{-n_1}
  =\iter f{n_2-n_1}.\nonumber
\end{align*}
This proves $\Cent(f)\subset\cyc f$ for every commuting germ,
without a restriction on its initial domain size or multiplier.
The reverse inclusion is \eqref{eq:integer-group}.

If $\iter f k=\iter f\ell$, comparison of derivatives gives
\begin{equation*}
 \lambda^k=\lambda^\ell
 \quad\Longrightarrow\quad\lambda^{k-\ell}=1
 \quad\Longrightarrow\quad k=\ell,
\end{equation*}
because $\alpha$ is irrational. This proves uniqueness.
\end{proof}

\subsubsection{The two arithmetic cases}

\begin{proof}[Proof of Corollary~\ref{cor:dichotomy}]
If $\alpha\notin\Bcal$, the established equivalence
\eqref{eq:BY} gives nonlinearizability of $Q_\alpha$.
Theorem~\ref{thm:main} gives the second line of
\eqref{eq:dichotomy}.

If $\alpha\in\Bcal$, choose an analytic linearizing coordinate
$\psi$. For $h\in\Cent(Q_\alpha)$, define
\begin{equation*}
 H=\psi\circ h\circ\psi^{-1},\qquad
 H(w)=\sum_{n\ge1}c_nw^n,
 \qquad c_1=h'(0)\ne0.
\end{equation*}
The commutation equation becomes
\begin{equation}\label{eq:rotation-coefficient-equation}
 H(\lambda w)=\lambda H(w)
 \quad\Longrightarrow\quad
 (\lambda^n-\lambda)c_n=0\quad(n\ge1).
\end{equation}
For $n\ge2$ the factor $\lambda^n-\lambda$ is nonzero.
Therefore
\begin{equation*}
 H(w)=c_1w=R_{c_1}(w).
\end{equation*}
Conversely, every $R_\mu$, $\mu\ne0$, commutes with
$R_\lambda$. The germ $\psi^{-1}R_\mu\psi$ is defined on a
sufficiently small neighbourhood of $0$, for every such $\mu$.
This proves the first line of \eqref{eq:dichotomy}.

In this first case, the centraliser contains uncountably many
multipliers $\mu\in\C^*$, whereas the multiplier set of the
integral iterates is the countable group
\begin{equation*}
 \{\lambda^k:k\in\Z\}.
\end{equation*}
Thus the first line does not give a trivial centraliser.
Combining the two cases proves \eqref{eq:iff-trivial}.
\end{proof}

\subsubsection{Minimality of the globalized parametrization}
The local injectivity of the map $\pi$ gives a sharper description
than the representation \eqref{eq:P-i-iterates} alone. Once algebraicity
is known, graph descriptions of invariant curves belong to the theory
of Medvedev--Scanlon~\cite{MS}, with a related statement in
Pakovich~\cite[Theorem~4.9]{Pakovich}. The elementary argument below
identifies the particular normalization already constructed, and
these algebraic-curve results are not used to infer algebraicity of
the initial analytic graph.

\begin{proposition}[Canonical globalized graph]\label{prop:canonical-graph}
In the construction used for Theorem~\ref{thm:main}, the normalized
projection exponents satisfy
\begin{equation}\label{eq:minimal-exponents}
 \min(n_1,n_2)=0.
\end{equation}
If $h=\iter f k$, then the globalized curve $\Gamma=\pi(X)$ is
\begin{equation}\label{eq:canonical-curve}
 \Gamma=
 \begin{cases}
  \{(z,\iter f k(z)):z\in\C\},& k\ge0,\\[1mm]
  \{(\iter f{-k}(w),w):w\in\C\},&k<0.
 \end{cases}
\end{equation}
In particular, $\Gamma$ is a smooth irreducible algebraic curve,
and one of its two normalized projections has degree one.
\end{proposition}
\begin{proof}
Suppose $n_1,n_2\ge1$. The two components of
$\pi\circ A^{-1}=(P_1,P_2)$ then factor through $f$:
\begin{equation*}
 (P_1(z),P_2(z))
 =\bigl(\iter f{n_1-1}(f(z)),\iter f{n_2-1}(f(z))\bigr).
\end{equation*}
At the critical point $c_*=-\lambda/2$,
\begin{equation*}
 f(c_*+t)=f(c_*-t)=-\frac{\lambda^2}{4}+t^2.
\end{equation*}
Hence for every sufficiently small $t\ne0$,
\begin{equation*}
 (P_1(c_*+t),P_2(c_*+t))
 =(P_1(c_*-t),P_2(c_*-t)),\qquad c_*+t\ne c_*-t.
\end{equation*}
This contradicts the local injectivity of $\pi\circ A^{-1}$,
which follows from \eqref{eq:pi-locally-injective} and the
biholomorphic nature of $A$. Thus \eqref{eq:minimal-exponents}
holds.

Since $k=n_2-n_1$, we have $(n_1,n_2)=(0,k)$ when $k\ge0$ and
$(n_1,n_2)=(-k,0)$ when $k<0$. Substituting into
\eqref{eq:P-i-iterates}, and using that $A(X)=\C$, gives
\eqref{eq:canonical-curve}. In the first case the curve has defining
equation
\begin{equation*}
 y-\iter f k(x)=0,
\end{equation*}
whose derivative in $y$ is $1$. In the second case its defining
equation is
\begin{equation*}
 x-\iter f{-k}(y)=0,
\end{equation*}
whose derivative in $x$ is $1$. The equations are irreducible,
since they are monic of degree one in the indicated variable, and
they also prove smoothness and the assertion about the degree-one
projection.
\end{proof}

\subsection{Formal commutants, convergence, and iterative roots}\label{sec:consequences}

\subsubsection{Multipliers of analytic commutants}
We first justify independently the unit-multiplier assertion used when
comparing conventions in the introduction. This is standard local
linearization theory rather than an additional arithmetic
classification.

\begin{lemma}[Unit multipliers in the nonlinearizable case]\label{lem:unit-multiplier}
Let $F\in\Diff$ have $F'(0)=\lambda$, with $|\lambda|=1$, and suppose
that $F$ is not analytically linearizable. Then
\begin{equation*}
 H\in\Cent(F)\quad\Longrightarrow\quad |H'(0)|=1.
\end{equation*}
\end{lemma}
\begin{proof}
Suppose $|H'(0)|\ne1$. Replacing $H$ by its inverse if necessary,
we obtain $K\in\Cent(F)$ with
\begin{equation*}
 b=K'(0),\qquad 0<|b|<1.
\end{equation*}
Koenigs' linearization theorem~\cite{Milnor} gives an invertible
analytic germ $\chi$ such that
\begin{equation}\label{eq:koenigs-commutant}
 \chi\circ K\circ\chi^{-1}=R_b.
\end{equation}
Put $T=\chi\circ F\circ\chi^{-1}=\lambda z+\sum_{n\ge2}t_nz^n$.
Since $KF=FK$, equations \eqref{eq:koenigs-commutant} imply
\begin{equation*}
 T(bz)=bT(z),\qquad (b^n-b)t_n=0\quad(n\ge2).
\end{equation*}
But $0<|b|<1$ gives $|b^{n-1}|<1$ for $n\ge2$, so $b^n-b\ne0$.
Consequently $T(z)=\lambda z$, which analytically linearizes $F$,
a contradiction.
\end{proof}

\subsubsection{The formal centraliser}
The formal coefficient calculation is standard for a nonresonant
one-dimensional germ, as in \cite{Milnor}. In the quadratic setting,
the uniqueness of a commutant with prescribed iterate multiplier also
appears in \cite[Lemma~3.4]{ACE}. The new convergence conclusions below
are consequences of Theorem~\ref{thm:main} rather than a new formal
linearization theorem.

Let $\Diffhat$ denote the group of formal series
$a_1z+\sum_{n\ge2}a_nz^n$ with $a_1\ne0$, and let
$\Centfor(F)$ denote the centraliser in that formal group.
Write $[z^n]S$ for the coefficient of $z^n$ in a series $S$.

\begin{lemma}[Formal linearization and the multiplier coordinate]\label{lem:formal}
Let
\begin{equation}\label{eq:formal-F}
 F(z)=az+\sum_{n\ge2}F_nz^n\in\Diffhat,
 \qquad a^n\ne1\quad(n\ge1).
\end{equation}
There is a unique normalized formal linearizing series
$L(z)=z+\sum_{n\ge2}\ell_nz^n$ satisfying $L\circ F=aL$.
For every $\mu\in\C^*$ there is exactly one formal commutant
with derivative $\mu$, namely
\begin{equation}\label{eq:formal-commutant-formula}
 H_\mu=L^{-1}\circ R_\mu\circ L.
\end{equation}
The derivative map identifies $\Centfor(F)$ with $\C^*$ and
\begin{equation*}
 H_\mu\circ H_\nu=H_{\mu\nu},\qquad
 H_\mu^{-1}=H_{\mu^{-1}}.
\end{equation*}
\end{lemma}
\begin{proof}
The coefficient equation for $L\circ F=aL$ at degree $n\ge2$ is
\begin{equation}\label{eq:formal-linearizer-recursion}
 (a^n-a)\ell_n
 =-F_n-\sum_{j=2}^{n-1}\ell_j[z^n]F(z)^j.
\end{equation}
The right-hand side depends only on $\ell_2,\ldots,\ell_{n-1}$,
and $a^n-a\ne0$ by \eqref{eq:formal-F}. Thus
\eqref{eq:formal-linearizer-recursion} determines $\ell_n$
uniquely at each degree, proving existence and uniqueness of $L$.
Its linear coefficient is one, so it has a unique formal inverse,
obtained recursively by solving $L\circ L^{-1}=z$.

Formula \eqref{eq:formal-commutant-formula} gives
\begin{align*}
 H_\mu\circ F
 &=L^{-1}\circ R_\mu\circ R_a\circ L\\
 &=L^{-1}\circ R_a\circ R_\mu\circ L
  =F\circ H_\mu.\nonumber
\end{align*}
It has derivative $\mu$. Conversely, if $H\in\Centfor(F)$,
then $S=L\circ H\circ L^{-1}$ satisfies
\begin{equation*}
 S(az)=aS(z),\qquad
 (a^n-a)[z^n]S=0\quad(n\ge1).
\end{equation*}
Thus $S(z)=\mu z$ for $\mu=H'(0)$, proving uniqueness and
\eqref{eq:formal-commutant-formula}. The group identities follow
by composing the dilations between $L$ and $L^{-1}$.
\end{proof}

\begin{corollary}[Exact convergence set]\label{cor:convergence}
Let $\alpha\notin\Bcal$ be irrational, let
$\lambda=e^{2\pi i\alpha}$, and let $f=Q_\alpha$.
For the unique formal commutant $H_\mu$ with multiplier $\mu\ne0$,
\begin{equation}\label{eq:exact-convergence-set}
 H_\mu\in\C\{z\}
 \quad\Longleftrightarrow\quad \mu\in\{\lambda^k:k\in\Z\}.
\end{equation}
If $\mu=\lambda^k$, then $H_\mu=\iter f k$. Otherwise the radius
of convergence of $H_\mu$ is zero.
\end{corollary}
\begin{proof}
If $H_\mu$ converges on a disc of positive radius, its nonzero
linear coefficient makes it an invertible analytic germ.
Its formal commutation identity is then an identity of convergent
germs. Hence
\begin{equation}\label{eq:convergence-necessity}
 H_\mu\in\Cent(f)=\cyc f
 \quad\Longrightarrow\quad H_\mu=\iter f k,
 \quad\mu=\lambda^k
\end{equation}
by Corollary~\ref{cor:dichotomy}. Conversely, each positive or
negative integral iterate is an analytic germ. Its derivative is
$\lambda^k$, and Lemma~\ref{lem:formal} identifies its Taylor
series with $H_{\lambda^k}$.
For any remaining $\mu$, a positive convergence radius would
contradict \eqref{eq:convergence-necessity}, so its radius must be
zero.
\end{proof}

\subsubsection{The explicit coefficient recursion}
For completeness, the formal commutants in
Corollary~\ref{cor:convergence} are determined by an explicit
triangular system. Write
\begin{equation*}
 H_\mu(z)=\sum_{n\ge1}b_nz^n,\qquad b_1=\mu.
\end{equation*}
The binomial identity
\begin{equation*}
 (\lambda z+z^2)^j
 =\sum_{s=0}^j\binom js\lambda^{j-s}z^{j+s}
\end{equation*}
gives
\begin{equation*}
 [z^n]H_\mu(f(z))
 =\lambda^n b_n+
   \sum_{j=\lceil n/2\rceil}^{n-1}
    \binom{j}{n-j}\lambda^{2j-n}b_j.
\end{equation*}
The other side of the commutation equation satisfies
\begin{equation*}
 [z^n]f(H_\mu(z))
 =\lambda b_n+\sum_{j=1}^{n-1}b_jb_{n-j}.
\end{equation*}
Subtracting yields, for every $n\ge2$,
\begin{equation*}
 (\lambda^n-\lambda)b_n
 =\sum_{j=1}^{n-1}b_jb_{n-j}
  -\sum_{j=\lceil n/2\rceil}^{n-1}
     \binom{j}{n-j}\lambda^{2j-n}b_j.
\end{equation*}
In particular,
\begin{align*}
 b_2&=\frac{\mu^2-\mu}{\lambda^2-\lambda},\\
 b_3&=\frac{2(\mu-\lambda)b_2}{\lambda^3-\lambda},\nonumber\\
 b_4&=\frac{(2\mu-3\lambda^2)b_3+b_2^2-b_2}
              {\lambda^4-\lambda}.\nonumber
\end{align*}
All denominators are nonzero. The recursion determines formal
coefficients for every $\mu\ne0$, whereas the convergence
classification \eqref{eq:exact-convergence-set} concerns the complete
infinite series rather than any finite truncation.

\subsubsection{Centralisers of nonzero iterates}

\begin{lemma}[Invariance under nonzero iteration]\label{lem:iterate-centre}
Let $F\in\Diff$ have a multiplier that is not a root of unity.
For every $N\in\Z\setminus\{0\}$,
\begin{equation*}
 \Cent(\iter FN)=\Cent(F).
\end{equation*}
\end{lemma}
\begin{proof}
If $H$ commutes with $F$, then it commutes with $F^{-1}$, since
multiplying $HF=FH$ by inverse germs gives $HF^{-1}=F^{-1}H$, and
induction then gives commutation with all nonzero iterates. This
proves the inclusion from right to left.

Conversely, let $H$ commute with $\iter FN$ and use the formal
linearizing series $L$ of $F$ from Lemma~\ref{lem:formal}. If
$a=F'(0)$, then
\begin{equation*}
 L\circ\iter FN\circ L^{-1}=R_{a^N}
\end{equation*}
also for negative $N$. Set $S=L\circ H\circ L^{-1}$. Its
commutation equation gives
\begin{equation*}
 S(a^Nz)=a^NS(z),\qquad
 (a^{Nn}-a^N)[z^n]S=0\quad(n\ge2).
\end{equation*}
Because $a$ is not a root of unity and $N\ne0$, all factors on
the left are nonzero. Therefore $S$ is a dilation and commutes
with $R_a$. Conjugating back by $L$ gives $HF=FH$ formally.
Both sides are convergent germs, so equality of their coefficients
is equality of analytic germs. This proves the other inclusion.
\end{proof}

\begin{corollary}\label{cor:iterate-centre}
If $f=Q_\alpha$ and $\alpha$ is non-Brjuno irrational, then
\begin{equation*}
 \Cent(\iter fN)=\cyc f\qquad(N\in\Z\setminus\{0\}).
\end{equation*}
For $|N|>1$, this group strictly contains $\cyc{\iter fN}$.
\end{corollary}
\begin{proof}
Combine Lemma~\ref{lem:iterate-centre} with
Corollary~\ref{cor:dichotomy}. The subgroup generated by
$\iter fN$ corresponds, under $\iter fk\mapsto k$, to
$N\Z\subset\Z$. For $|N|>1$ this is a proper subgroup.
\end{proof}

\begin{corollary}[Analytic roots of iterates]\label{cor:roots}
Let $f=Q_\alpha$ with $\alpha$ non-Brjuno irrational.
For integers $m\ge1$ and $N\ne0$, the equation
\begin{equation}\label{eq:root-equation}
 \iter hm=\iter fN,\qquad h\in\Diff,
\end{equation}
has a solution if and only if $m\mid N$. When it has a solution,
that solution is unique and equals $\iter f{N/m}$.
In particular, $f$ has no analytic iterative root of order $m\ge2$.
\end{corollary}
\begin{proof}
An $h$ satisfying \eqref{eq:root-equation} commutes with its own
$m$th iterate, so
\begin{equation*}
 h\circ\iter fN=h\circ\iter hm
   =\iter hm\circ h=\iter fN\circ h.
\end{equation*}
Corollary~\ref{cor:iterate-centre} gives $h=\iter fk$ for a
unique integer $k$. Differentiating \eqref{eq:root-equation}
at $0$ now yields
\begin{equation*}
 \lambda^{mk}=\lambda^N
 \quad\Longrightarrow\quad \lambda^{mk-N}=1
 \quad\Longrightarrow\quad mk=N.
\end{equation*}
Thus $m\mid N$ and $k=N/m$. Conversely, if $m\mid N$, then
\begin{equation*}
 \bigl(\iter f{N/m}\bigr)^{\circ m}=\iter fN.
\end{equation*}
Taking $N=1$ proves the final assertion.
\end{proof}

\begin{remark}
The hypothesis $N\ne0$ in Corollary~\ref{cor:roots} is necessary
for its formulation without a commutation assumption: the equation
$h^{\circ m}=\id$ has nonidentity finite-order analytic solutions.
Such germs do not commute with the nonlinearizable quadratic unless
they are the identity.
\end{remark}

\subsubsection{Noninvertible commuting germs}

\begin{corollary}\label{cor:noninvertible}
Let $f=Q_\alpha$ with $\alpha$ non-Brjuno irrational. If a
holomorphic germ $h$ satisfies $h(0)=0$ and $h\circ f=f\circ h$,
then either $h\equiv0$ or $h=\iter fk$ for some $k\in\Z$.
\end{corollary}
\begin{proof}
If $h\not\equiv0$, write
\begin{equation*}
 h(z)=bz^r+O(z^{r+1}),\qquad b\ne0,\quad r\ge1.
\end{equation*}
The leading terms in the two compositions are
\begin{equation*}
 h(f(z))=b\lambda^rz^r+O(z^{r+1}),\qquad
 f(h(z))=\lambda bz^r+O(z^{r+1}).
\end{equation*}
Thus $\lambda^{r-1}=1$, which forces $r=1$.
The inverse function theorem gives $h\in\Diff$, and
Corollary~\ref{cor:dichotomy} applies.
\end{proof}

\begin{remark}[Why the quadratic degree matters]\label{rem:quadratic-degree}
The argument uses the degree two in two separate places. First,
Proposition~\ref{prop:plane} gives only $\deg G\ge2$ before the
projections are controlled. The quadratic target then permits
\begin{equation*}
 \overline{\disc{cr^2}}\subset G(\disc r),\qquad
 H(cr^2)\le H(r)^2,
\end{equation*}
which forces polynomial growth. For a target of higher degree $d$, the corresponding right-hand
power would be $d$, so the lower bound $\deg G\ge2$ alone would no
longer give the same elementary growth conclusion. This limitation
concerns the displayed estimate rather than the scope of Eremenko's
more general theorem~\cite{Eremenko}.
Second, after polynomial realization, a quadratic polynomial is
already a minimal-degree nonlinear element of its polynomial
commutant. This is exactly \eqref{eq:minimal-degree-two}.

A composite-degree example also shows directly that the conclusion
cannot simply be stated for all nonlinearizable polynomial germs.
\end{remark}

\begin{example}[A nonlinearizable quartic with an extra commutant]\label{ex:quartic}
Choose $\beta\notin\Bcal$, put $\eta=e^{2\pi i\beta}$, and define
\begin{equation*}
 R(z)=\eta z+z^2,\qquad
 P(z)=R^{\circ2}(z)
 =\eta^2z+(\eta+\eta^2)z^2+2\eta z^3+z^4.
\end{equation*}
Then $P$ is nonlinearizable at $0$ and $R\in\Cent(P)$, but
$R\notin\cyc P$.
\end{example}
\begin{proof}
Commutation follows from associativity of iteration. If $P$ were
linearizable by $\psi$, then $S=\psi R\psi^{-1}$ would commute
with $R_{\eta^2}$. Since $\eta^2$ is not a root of unity,
the coefficient argument in
\eqref{eq:rotation-coefficient-equation} would give
\begin{equation*}
 S(z)=\eta z.
\end{equation*}
This would linearize $R$, contrary to \eqref{eq:BY} and
$\beta\notin\Bcal$. Thus $P$ is nonlinearizable.
If $R=P^{\circ k}$ as germs for an integer $k$, their derivatives
would satisfy
\begin{equation*}
 \eta=\eta^{2k}\quad\Longrightarrow\quad\eta^{2k-1}=1,
\end{equation*}
which is impossible for irrational $\beta$.
\end{proof}

\begin{remark}[Scope of the main theorem]\label{rem:scope-coda}
The result is restricted to the quadratic family and to local analytic
centralisers at the indicated fixed point. It does not assert cyclicity
for all nonlinearizable germs, nor does it settle optimal simultaneous
linearization conditions for arbitrary commuting families. The
formal-convergence and root classifications follow from this specific
quadratic rigidity statement.
\end{remark}


\begin{thebibliography}{99}

\bibitem{Ahlfors}
L.~V.~Ahlfors,
\emph{Complex Analysis}, 3rd ed., McGraw--Hill, New York, 1979.

\bibitem{ACE}
A.~Avila, D.~Cheraghi and A.~Eliad,
\emph{Analytic maps of parabolic and elliptic type with trivial centralisers},
\emph{Ann. Inst. H. Poincar\'e C Anal. Non Lin\'eaire} \textbf{39} (2022), no.~4,
885--903.
\href{https://doi.org/10.4171/AIHPC/22}{doi:10.4171/AIHPC/22}.

\bibitem{Biswas}
K.~Biswas,
\emph{Complete conjugacy invariants of nonlinearizable holomorphic dynamics},
\emph{Discrete Contin. Dyn. Syst.} \textbf{26} (2010), no.~3, 847--856.
\href{https://doi.org/10.3934/dcds.2010.26.847}{doi:10.3934/dcds.2010.26.847}.
Also available as \href{https://arxiv.org/abs/0903.2394}{arXiv:0903.2394}.

\bibitem{BC}
X.~Buff and A.~Ch\'eritat,
\emph{A new proof of a conjecture of Yoccoz},
\emph{Ann. Inst. Fourier (Grenoble)} \textbf{61} (2011), no.~1, 319--350.
\href{https://doi.org/10.5802/aif.2603}{doi:10.5802/aif.2603}.

\bibitem{Cheraghi}
D.~Cheraghi,
\emph{Topology of irrationally indifferent attractors},
\emph{Ann. Sci. \'Ec. Norm. Sup\'er. (4)} \textbf{58} (2025), no.~6,
1319--1381.
\href{https://doi.org/10.24033/asens.2630}{doi:10.24033/asens.2630}.
Also available as \href{https://arxiv.org/abs/1706.02678}{arXiv:1706.02678}.

\bibitem{Chirka}
E.~M.~Chirka,
\emph{Complex Analytic Sets},
Mathematics and its Applications (Soviet Series), vol.~46,
Kluwer Academic Publishers, Dordrecht, 1989.

\bibitem{Eremenko}
A.~E.~Eremenko,
\emph{On some functional equations connected with iteration of rational functions},
\emph{Algebra i Analiz} \textbf{1} (1989), no.~4, 102--116 (Russian);
English translation, \emph{Leningrad Math. J.} \textbf{1} (1990), no.~4, 905--919.
\href{https://www.math.purdue.edu/~eremenko/dvi/feq-e.pdf}{Author's copy of the English translation}.

\bibitem{Forster}
O.~Forster,
\emph{Lectures on Riemann Surfaces},
Graduate Texts in Mathematics, vol.~81, Springer-Verlag, New York, 1981.
\href{https://doi.org/10.1007/978-1-4612-5961-9}{doi:10.1007/978-1-4612-5961-9}.

\bibitem{Inou}
H.~Inou,
\emph{Extending local analytic conjugacies},
\emph{Trans. Amer. Math. Soc.} \textbf{363} (2011), no.~1, 331--343.
\href{https://doi.org/10.1090/S0002-9947-2010-05049-4}{doi:10.1090/S0002-9947-2010-05049-4}.

\bibitem{JX}
Z.~Ji and J.~Xie,
\emph{Local rigidity of Julia sets},
\emph{Amer. J. Math.}, to appear.
\href{https://arxiv.org/abs/2302.02562v2}{arXiv:2302.02562v2 (2025)}.
Theorem numbers cited in the text refer to this version.

\bibitem{MS}
A.~Medvedev and T.~Scanlon,
\emph{Invariant varieties for polynomial dynamical systems},
\emph{Ann. of Math. (2)} \textbf{179} (2014), no.~1, 81--177.
\href{https://doi.org/10.4007/annals.2014.179.1.2}{doi:10.4007/annals.2014.179.1.2}.

\bibitem{Milnor}
J.~Milnor,
\emph{Dynamics in One Complex Variable}, 3rd ed.,
Annals of Mathematics Studies, vol.~160,
Princeton University Press, Princeton, NJ, 2006.

\bibitem{Moser}
J.~Moser,
\emph{On commuting circle mappings and simultaneous Diophantine approximations},
\emph{Math. Z.} \textbf{205} (1990), no.~1, 105--121.
\href{https://doi.org/10.1007/BF02571227}{doi:10.1007/BF02571227}.

\bibitem{Pakovich}
F.~Pakovich,
\emph{Polynomial semiconjugacies, decompositions of iterations, and invariant curves},
\emph{Ann. Sc. Norm. Super. Pisa Cl. Sci. (5)} \textbf{17} (2017), no.~4,
1417--1446.
\href{https://pakovich.github.io/sp.pdf}{Author's copy}.

\bibitem{PM}
R.~P\'erez Marco,
\emph{Nonlinearizable holomorphic dynamics having an uncountable number of symmetries},
\emph{Invent. Math.} \textbf{119} (1995), 67--127.
\href{https://doi.org/10.1007/BF01245175}{doi:10.1007/BF01245175}.

\bibitem{Ritt}
J.~F.~Ritt,
\emph{Permutable rational functions},
\emph{Trans. Amer. Math. Soc.} \textbf{25} (1923), no.~3, 399--448.
\href{https://doi.org/10.1090/S0002-9947-1923-1501252-3}{doi:10.1090/S0002-9947-1923-1501252-3}.

\bibitem{Yoccoz}
J.-C.~Yoccoz,
\emph{Th\'eor\`eme de Siegel, nombres de Bruno et polyn\^omes quadratiques},
\emph{Ast\'erisque} \textbf{231} (1995), 3--88.

\end{thebibliography}
\end{document}